\documentclass[10pt]{article}

\usepackage{amsmath}
\usepackage{array}
\usepackage{amsfonts}
\usepackage{amssymb}
\usepackage{amsthm}
\usepackage{epsfig,graphics,graphicx}
\usepackage{stmaryrd}

\usepackage{hyperref}

\usepackage{tikz}
\usepackage{wrapfig} 
\usepackage[all]{xy}

\usepackage[babel=true]{csquotes}

  \theoremstyle{plain}
\newtheorem{theorem}{Theorem}[section]
\newtheorem{proposition}[theorem]{Proposition}
\newtheorem{corollary}[theorem]{Corollary}
\newtheorem{lemma}[theorem]{Lemma} %%%%%
  \theoremstyle{remark}
\newtheorem{remark}[theorem]{Remark}
  \theoremstyle{definition}
\newtheorem{definition}[theorem]{Definition}
\newtheorem*{remarks}{Remarks}
\newtheorem*{note}{Note}

\newcommand{\ie}{\textit{i.e. }}
\newcommand{\G}{\mathcal{G}}
\newcommand{\NN}{\mathbb{N}}
\newcommand{\ZZ}{\mathbb{Z}}
\newcommand{\RR}{\mathbb{R}}
\renewcommand{\SS}{\mathbb{S}^1}
\renewcommand{\i}{\iota}

\title{Gauss diagrams of real and virtual knots in the solid torus}
\author{Arnaud Mortier \\ \itshape mortier@math.ups-tlse.fr}
\date{\today}

\begin{document}

\maketitle

\begin{abstract}
\footnotesize
Gauss diagrams in knot theory were introduced by Polyak and Viro (\cite{PV}) as an appropriate device to describe finite type invariants. As a by-product, they naturally gave rise to the fruitful theory of virtual knots, introduced and developed by Kauffman (\cite{K1}, \cite{K2}).

The purpose of this article is to define a new type of Gauss diagrams, adapted from the decorated diagrams introduced by Fiedler (\cite{F1}, see also \cite{F2}) to describe knots in the solid torus with projections in $\RR\times\SS$. We see that it provides an efficient tool for showing that a knot diagram can be fully recovered from its decorated Gauss diagram, and we use it to establish a characterization of the decorated Gauss diagrams of closed braids.
\end{abstract}

\tableofcontents

\section{Introduction: what is a Gauss diagram?}
\label{sec:Introduction}

We define three versions of Gauss diagram theories, beginning with the classical settings and gradually refining them, and we study their basic properties.
When Gauss diagrams are considered as topological objects (for instance when we look at their first homology group), remember that the arrows only \textit{look like} they intersect.
The word \enquote{real} will be used as the opposite of \enquote{virtual} - we save the word \enquote{classical} for knot diagrams in $\RR^2$. The author deeply apologizes to the reader who is used to real knots in physics terminology.

\subsection{The classical case and the birth of virtual knot theory}

A \textit{classical Gauss diagram} is an oriented circle in which a finite number of couples of points are linked by an abstract, signed and oriented arrow. Starting with a classical knot diagram $D$ in $\RR^2$, one obtains the \textit{associated Gauss diagram} by considering a parametrization of $D$ by an oriented circle, and connecting the preimages of each crossing by an arrow oriented from the underpassing to the overpassing point, given as a sign the writhe number of the crossing (see Fig.\ref{1}).

\begin{figure}[h]
\centering 
\psfig{file=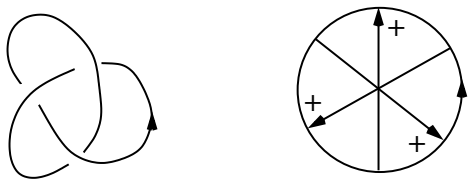,scale=0.9}
\caption{A classical Gauss diagram}\label{1}
\end{figure}

Now a natural question is: \enquote{is any classical Gauss diagram associated to some knot?}, and the answer is no. The simplest example is pictured on \mbox{Fig.\ref{2}:} try to draw a corresponding knot diagram, you will soon find it necessary to add a crossing where no arrow allows it.

\begin{figure}[h!]
\centering 
\psfig{file=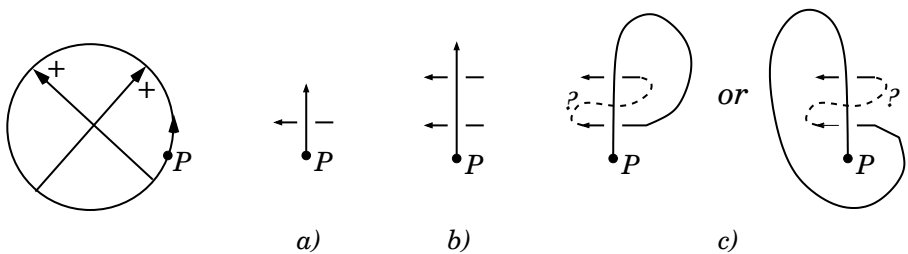}
\caption{This one cannot come from a knot}\label{2}
\end{figure}

This is how virtual knot theory starts: add whichever crossings you need to complete the picture, and draw a circle around them, to notify that these are not regular crossings. Those so-called \textit{virtual} crossings are subject to a new set of Reidemeister moves, precisely those which leave the underlying Gauss diagram unchanged (in particular, the last move depicted on Fig.\ref{3} \mbox{is forbidden !).}

\begin{figure}[h!]
\centering 
\psfig{file=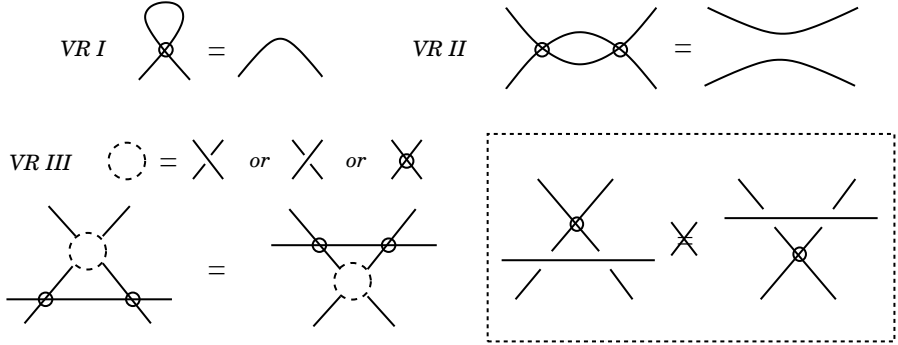,scale=0.8}
\caption{Virtual Reidemeister moves}\label{3}
\end{figure}
\noindent This set of \textit{virtual Reidemeister moves} is satisfactory because of the following:

\begin{lemma} [Fundamental property of virtual crossings] \label{fund}
Any two knot diagrams which differ only between two points, such that the two different arcs are homotopic and contain only virtual crossings, are equivalent under virtual Reidemeister moves.
\end{lemma}

This property means that the only relevant datum in a virtual arc is where it is from, and where it is going to, regardless of what it does in the meantime. In other words, virtual knot diagrams up to virtual Reidemeister moves contain precisely as much information as their Gauss diagrams do:

\begin{corollary}
Two virtual knot diagrams are equivalent under virtual Reidemeister moves if and only if they have the same Gauss diagram.
\end{corollary}

\subsection{Decorated Gauss diagrams of knots in the solid torus}

From now on, a knot diagram will be a virtual knot diagram in $\RR\times\SS$. Although we usually speak about virtual knot theory in the solid torus, these are \textbf{not} projections of anything living in a solid torus - unless there is no virtual crossing. We recall the definition of a decorated Gauss diagram given by Fiedler in this context:

\begin{definition}\label{Fiedler}
A \textit{decorated Gauss diagram} is defined as a classical Gauss diagram with the additional datum of a \textit{valuation} - that is a signed integer - to each arrow, and to the whole circle itself.\\ Let $A$ be an arrow in a Gauss diagram. The loop which goes along the orientation of the knot from the head of $A$ to its tail, and then back to the head along $A$ (see Fig.\ref{4}) is called the \textit{distinguished loop} associated to $A$.
Similarly, if $c$ is a real crossing of a virtual knot, then the loop which goes from the overpassing point to the underpassing, along the orientation of the knot, is called the \textit{distinguished loop} associated to $c$.
\end{definition}

\begin{figure}[h!]
\centering 
\psfig{file=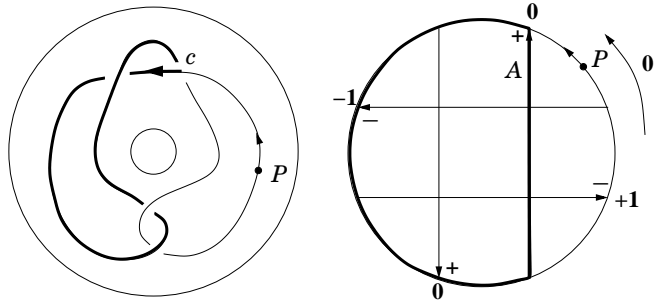,scale=1}
\caption{The decorated Gauss diagram of a knot, and a distinguished loop}\label{4}
\end{figure}

One can still associate a decorated Gauss diagram to a knot diagram: repeat the classical construction, then define the valuations of the arrows as the homology classes in $H_1\left( \RR\times\SS\right) $ of the distinguished loops of the corresponding crossings. The valuation of the circle is the homology class of the whole knot.

\begin{proposition}\label{basic} Every decorated Gauss diagram is represented by a virtual knot.
\end{proposition}
\begin{proof} By induction on the number of arrows:
A diagram with no arrow is just the datum of an integer and may be represented by any totally virtual knot with the required homology class
(actually there is only one such knot, up to virtual Reidemeister moves).
Take an $n$-arrow decorated Gauss diagram, forget one of its arrows, say $A$, and represent the remaining diagram by a knot.
Take two little neighborhoods of what should be the endpoints of $A$ on the knot, and homotope one of them to come near the other and cross it once, with the required writhe,
declaring virtual any crossing added in the process. Making it turn around the circle an appropriate number of times before performing the crossing provides control on the valuation (see Fig.\ref{5}, where dashed parts of the knot have only virtual intersections). Finally, notice that this operation does not affect the $n-1$ other valuations.
\end{proof}

\begin{figure}[h!]
\centering 
\psfig{file=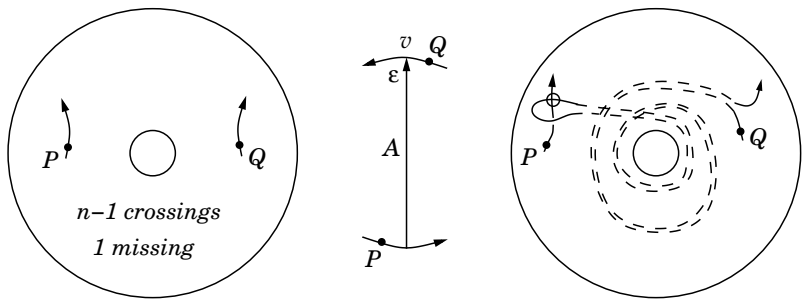,scale=1}
\caption{Adjusting all the parameters of an arrow}\label{5}
\end{figure}

\begin{lemma}\label{hom} Let $\gamma$ be a loop in a decorated Gauss diagram $G$. The homology class of the corresponding loop in $\RR\times\SS$ does not depend on the knot representing $G$.
\end{lemma}

\begin{proof} From the topological viewpoint, a Gauss diagram with $n$ arrows has the homotopy type of the wedge of $n+1$ circles. Its first homology group is generated by the class of the circle, plus the classes of the distinguished loops associated to the arrows. Thus the information contained in a decorated Gauss diagram determines the image of a homology basis, therefore the image of any loop.
\end{proof}

\subsection{Tangles and $T-$diagrams}

\begin{definition}
A \textit{(virtual) tangle diagram} is an oriented uni- and tetra-valent graph properly embedded in $\RR\times\left[ 0,1\right]$, where each tetravalent vertex has been decorated as a real or virtual crossing. Under the identification of $\RR\times\left\lbrace 0\right\rbrace$ with $\RR\times\left\lbrace 1\right\rbrace$, such a diagram becomes a link diagram in $\RR\times\SS$, and throughout the article we will always assume that our tangles close into \textit{knots}.
Since we look at diagrams rather than any kind of equivalence classes, this definition is equivalent to virtual knot diagrams in $\RR\times\SS$ that are transverse to some specified section $S=\RR\times \left\lbrace t\right\rbrace$. We will refer to either point of view without distinction.
\end{definition}

\begin{definition} A \textit{T-diagram} is a decorated Gauss diagram together with an additional decoration by signed markings (as shown on Fig.\ref{6}), required to be away from the arrows,
such that the valuation of any arrow is equal to the sum of the markings met by its distinguished loop.
It is said to be \textit{positive} if it has at least one marking and all of them are positive, and textit{non-negative} if it is positive or has no marking.

A $T-$diagram which becomes $G$ when we forget about its markings is called a \textit{refinement} of $G$.
\end{definition}

\begin{figure}[h!]
\centering 
\psfig{file=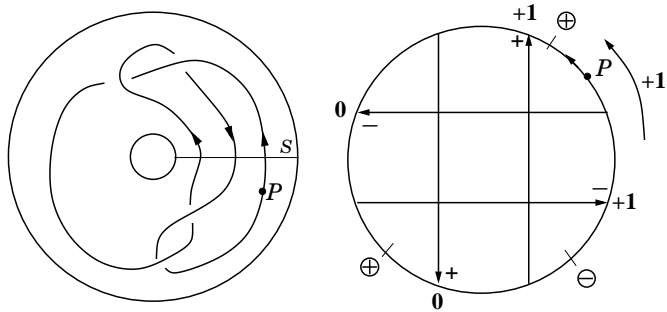,scale=1}
\caption{Associating a $T-$diagram to a knot and a section}\label{6}
\end{figure}

It is straightforward to see that if $D\subset \RR \times \SS$ is a virtual knot diagram with decorated Gauss diagram $G$, then every section $S= \RR\times \left\lbrace t\right\rbrace$ in general position with $D$ naturally defines a refinement of $G$, using intersection numbers between $S$ and the knot as markings (see Fig\ref{6}).

\begin{remarks}

$\triangleright$ As a $0$-cycle, the set of markings of a $T-$diagram is Poincar\'e-dual to the $1$-cocycle which takes a loop to its homology class in $\RR\times\SS$.

$\triangleright$ Any decorated Gauss diagram admits a refinement, since it may be represented by a knot (Lemma \ref{basic}).

$\triangleright$ In section \ref{sec:virtual} (Proposition \ref{rmov}), we will see a set of elementary moves with which we may pass from any refinement to any other.

\end{remarks}

The following proposition gives a hint on the usefulness of this notion for our purposes. Compare with Proposition \ref{basic}.

\begin{lemma}\label{Trep}
Every $T-$diagram is represented by a virtual tangle.
\end{lemma}
\begin{proof}
Draw a section $S=\RR\times\left\lbrace t\right\rbrace$ in $\RR\times\SS$. Then for each marking on our $T-$diagram, draw anywhere on $S$ the local behaviour that a representing tangle should have. Similarly, for each arrow draw a crossing anywhere in the picture, with the indicated writhe. Then, without creating any more intersection with $S$, connect all these little parts together in the required order, declaring virtual any additional crossing needed.
\end{proof}

Notice that the construction we just described contains only two choices: 

$\triangleright$ The ordering of the set of markings of $\G$ (we need it when we draw them onto $S$, which is $1-$dimensional). But any two different choices at this stage are related by virtual Reidemeister II moves, as shown on Fig.\ref{n+k}.

\begin{figure}[h!]
\centering 
\psfig{file=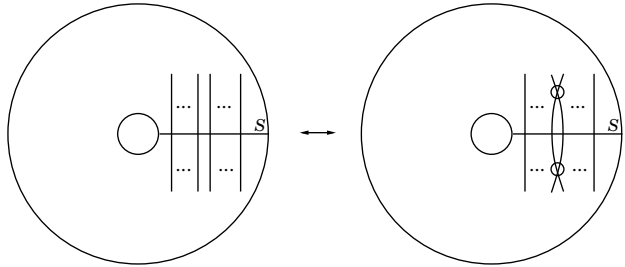,scale=0.9}
\caption{If we can make a transposition, then we can make any permutation}\label{n+k}
\end{figure}

$\triangleright$ The position of the virtual arcs in the last step of the construction. But the fundamental property of virtual crossings literally says that this last choice has no influence either, up to virtual Reidemeister moves.

So we actually proved much better:

\begin{proposition}\label{Tunique}
Two tangles with the same $T-$diagram are equivalent under virtual Reidemeister moves.
\begin{flushright}
$\square$
\end{flushright}
\end{proposition}

The \enquote{real} version of Lemma \ref{Trep} is false, in the sense that a $T-$diagram whose underlying decorated Gauss diagram is known to be represented by a real knot (\textit{i.e.} without virtual crossing) may not necessarily be represented by a real tangle. Though Lemma \ref{key} will show that this is not so far from being true.

\section{Decorated Gauss diagrams characterize virtual knots}
\label{sec:virtual}

In the classical case, it is easy to see that knot diagrams up to virtual Reidemeister moves are the same as Gauss diagrams - it is contained in \mbox{Proposition \ref{Tunique}.}\\ In $\RR\times\SS$, there is homology which we need to control: this is what $T-$diagrams do. Unfortunately, a given decorated Gauss diagram may have many refinements, so we first need to understand how they are linked with each other.

\begin{definition}
Adding or deleting a couple of markings of opposite signs right next to each other in a $T-$diagram is called an \textit{elementary refinement move of type I}.
If $A$ is an arrow, then adding a positive (\textsl{resp.} negative) marking just after each endpoint of $A$, and adding a negative (\textsl{resp.} positive) marking just before them, is called an \textit{elementary refinement move of type II+ (\emph{resp.} II-)} (see Fig.\ref{7}). Note that the reverse operation of a move of type II+ is just a move of type II- followed by four deletions of type I.
\end{definition}

\begin{figure}[h!]
\centering 
\psfig{file=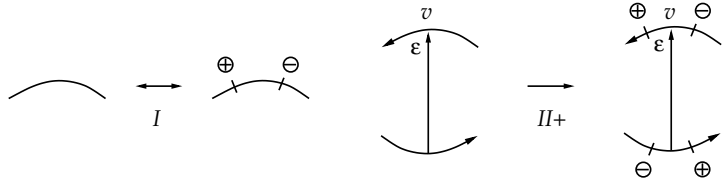,scale=1}
\caption{Elementary refinement moves of both types}\label{7}
\end{figure}

\begin{proposition}\label{rmov}
Any two refinements of a decorated Gauss diagram are linked by a finite number of elementary refinement moves of type I and II+.
\end{proposition}

\begin{proof} Let $\G_1$ and $\G_2$ have the same underlying decorated Gauss diagram $G$. We call \textit{edges} of a Gauss diagram the connected components of the complementary of its arrows.
Up to elementary refinement moves of type I, we may assume that any edge of either $\G_1$ or $\G_2$ contains markings with all the same sign. Two such diagrams may then be assimilated to elements of the $\ZZ-$module freely generated by the edges of $G$, numbered from $e_1$ to $e_{2n}$, where $n$ denotes the number of arrows in $G$.

Notice that if $n\leq 1$, then $G$ only has one refinement up to type I moves. Now we assume $n\geq 2$.

Let $M\in \mathcal{M}_{n+1,2n}(\ZZ)$ be a matrix whose arrows are defined as the distinguished loops of $G$, written in the basis $\left\lbrace e_i\right\rbrace$. We look at $M$ as a linear map $\RR^{2n}\rightarrow\RR^{n+1}$.
With this description, the \enquote{difference} between $\G_1$ and $\G_2$ is nothing but an element of $\ZZ^{2n}\cap\text{Ker}M$ - a way of changing the integer assigned to each edge of $G$ without changing their sum on any distinguished loop. Since the distinguished loops of $G$ form a basis of $H_1(G)$, $M$ has maximal rank, \ie $\text{dim}(\text{Ker}M)=n-1$. We are now going to prove that the refinement moves associated to any $n-1$ arrows of $G$ form a suitable basis of $\text{Ker}M$.

Let $A_n$ be any arrow of $G$. At least one of the $3$ or $4$ edges of $G$ that are adjacent to $A_n$ is also adjacent to another arrow $A_{n-1}$. Erase $A_{n-1}$: in the new Gauss diagram (if $n\geq 3$) there must be again an arrow $A_{n-2}$ with an adjacent edge in common with $A_n$. This means that in the original diagram $G$, it had an adjacent edge which was also adjacent to either $A_n$ or $A_{n-1}$, but certainly not any other arrow. Iterating this process, we get the arrows of $G$ into some order $\left( A_i\right)_{i=1}^n$ satisfying the following: if $x_i\in \text{Ker}M$ corresponds to the type II+ move associated to $A_i$, then
\[
\forall i \in \left\lbrace 1,\ldots ,n-1\right\rbrace , \exists j\in \llbracket 1,2n\rrbracket \mid \left\lbrace \begin{array}{l}
<x_i,e_j>\in \left\lbrace -1,1 \right\rbrace \\
\forall k \in \left\lbrace 1,\ldots ,i-1\right\rbrace , <x_k,e_j>=0
\end{array}
\right. 
.
\]

This implies not only that $\left\lbrace x_i \right\rbrace$ forms a basis of Ker$M$, but also that any element of $\ZZ^{2n}\cap\text{Ker}M$ has integer coordinates in that basis. So we may go from $\G_1$ to $\G_2$ by means of elementary moves of type I, II+ and II-. Finally, notice that the elementary move of type II- associated to some arrow is the sum of the type II+ moves associated to all the others.
\end{proof}

\begin{lemma}\label{perf}
Let $(D,S)$ be a virtual tangle diagram in $\RR\times\SS$ and $\G$ its $T-$diagram. If $\G^\prime$ is obtained from $\G$ by one elementary refinement move, then there is a virtual knot diagram $D^\prime$ equivalent to $D$ under virtual Reidemeister moves such that $\G^\prime$ is the $T-$diagram of $\left( D^\prime,S\right)$.
\end{lemma}
\begin{proof}
Fig.\ref{n} shows how to perform type I moves. Again a dashed arc represents a part of the knot with only virtual intersections. Both parts of the proof are fully based on the fundamental property of virtual \mbox{crossings (Lemma \ref{fund}).}

\begin{figure}[h!]
\centering 
\psfig{file=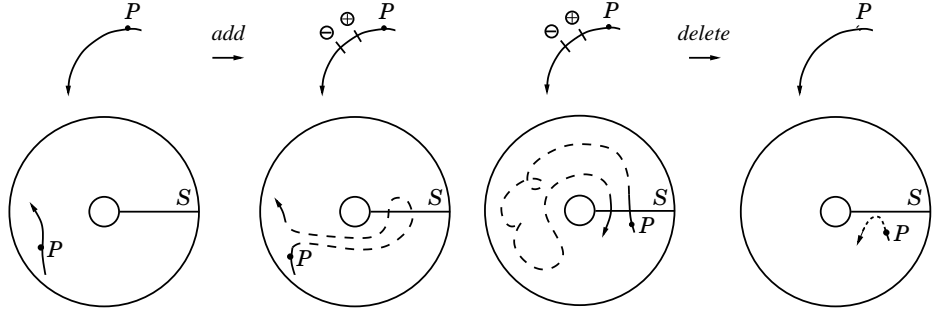,scale=1}
\caption{Performing elementary moves of Type I}\label{n}
\end{figure}

To perform a type II move, we begin with two moves of type I (Fig.\ref{n+1}, step 1). The two internal branches are chosen so close that they both (virtually) intersect the same other branches of the knot, which allows us to slide the real crossing, using virtual Reidemeister III move from Fig.\ref{3}, all the way to the other side of $S$.

\begin{figure}[h!]
\centering 
\psfig{file=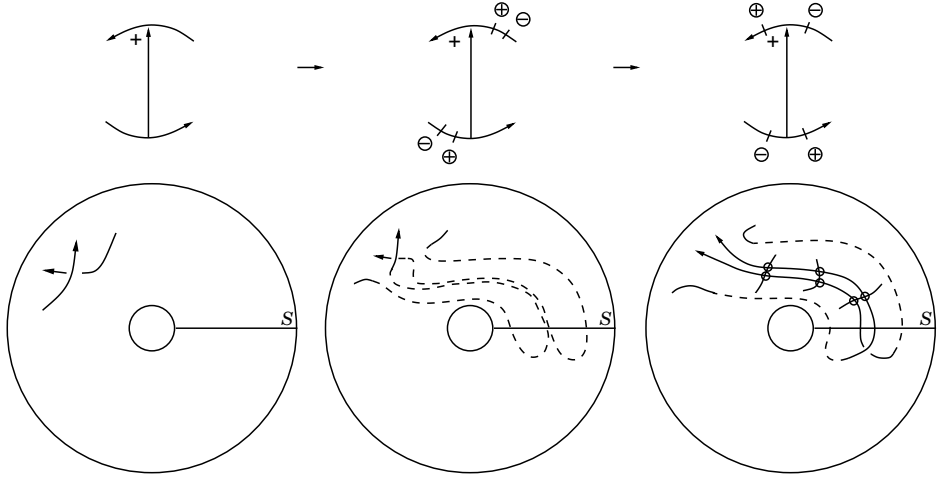,scale=1}
\caption{Performing an elementary move of Type II}\label{n+1}
\end{figure}
\end{proof}

We are now in a position to prove the announced theorem.

\begin{theorem}\label{virtual}
Two virtual knot diagrams in $\RR\times\SS$ with the same decorated Gauss diagram are equivalent under virtual Reidemeister moves.
\end{theorem}

\begin{proof}
Proposition \ref{rmov} and Lemma \ref{perf} together imply that the common decorated Gauss diagram $G$ has a refinement $\G$ such that up to virtual Reidemeister moves, both knot diagrams write as the closure of a tangle with $T-$diagram $\G$. Now Proposition \ref{Tunique} concludes the proof.
\end{proof}

\section{A stronger version in the case of real knots}
\label{sec:real}

In this section we show that there is a version of Theorem \ref{virtual} without Reidemeister moves, as soon as the knot diagrams are real and cannot be isotoped into a little disc. This case we need to avoid actually amounts to the classical theory, where such a theorem holds only up to \textit{real} Reidemeister moves. So decorated Gauss diagrams of knots in the solid torus contain more information about the knots they represent than classical Gauss diagrams do (Corollary \ref{wind}).

Let us describe a little bit those diagrams we are looking at:
\begin{definition}
We say that a real knot diagram is \textit{full} if it may not be isotoped into a little disc.
\end{definition}

\begin{figure}[h!]
\centering 
\psfig{file=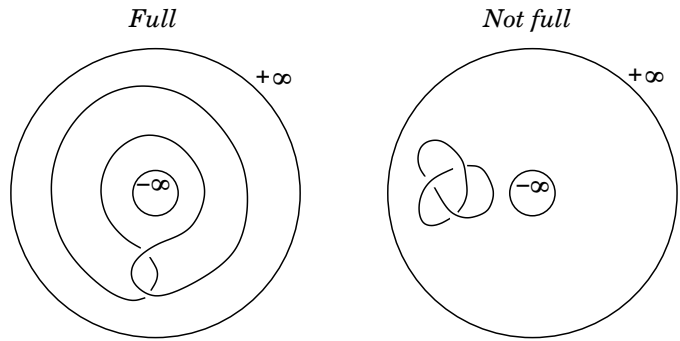,scale=1}
\caption{Fullness is not defined up to Reidemeister moves}\label{n+k+1}
\end{figure}

Let $D$ be a full diagram, and consider the complementary of $D$ in $\RR\times \SS$. It has two connected components with the homotopy type of a circle, each of which is bounded by a simple loop in $D$ with homology class $1$, and every other component is homeomorphic to an open disc - bounded by a simple loop in $D$ with homology class $0$.
\begin{definition}
In a full knot diagram, the two loops with homology class $1$ we just described are called the \textit{leftmost} and \textit{rightmost} loops, according to whether they bound the $-\infty$ or $+\infty$ end of $\RR\times\SS \setminus D$. The boundaries of the disc components are called the \textit{internal loops} of $D$.
\end{definition}

\begin{lemma}\label{full}
A real knot diagram is full if and only if its decorated Gauss diagram has at least one valuation different from $0$.
\end{lemma}
\begin{proof}
Recall that the valuations of a decorated Gauss diagram are the images of a basis of $H_1(G)$ into $H_1(\RR \times \SS)$. If $D$ is not full then clearly every loop in its Gauss diagram has homology class $0$ in $\RR \times \SS$.
Conversely, if $D$ is full, then its leftmost loop has homology class $1$. Since by assumption $D$ is real, this loop actually corresponds to a loop in $G$, because the two branches of any crossing are actually connected by an arrow. Since some loop has a nontrivial image, any basis must contain an element with a nontrivial image.
\end{proof}

The key ingredient in Theorem \ref{real} is the following (recall that an isotopy of a knot diagram does not involve Reidemeister moves):

\begin{lemma}\label{key}
Let $D$ be a real knot diagram, and $\G$ any refinement of its decorated Gauss diagram. Then there is a refinement $\mathcal{G}^\prime$ obtained from $\mathcal{G}$ by removing some markings, a knot diagram $D^\prime$ isotopic to $D$ and a section $S=\RR\times \left\lbrace t\right\rbrace$, such that $\mathcal{G}^\prime$ is the $T-$diagram associated to $\left( D^\prime, S\right)$.
\end{lemma}

\begin{proof}
If $D$ is isotopic to a knot diagram $D^\prime$ contained in a little disc, then put $\mathcal{G}^\prime=\mathcal{G}$ with all the markings removed, and choose a section $S$ that avoids $D^\prime$.

\begin{figure}[h!]
\centering 
\psfig{file=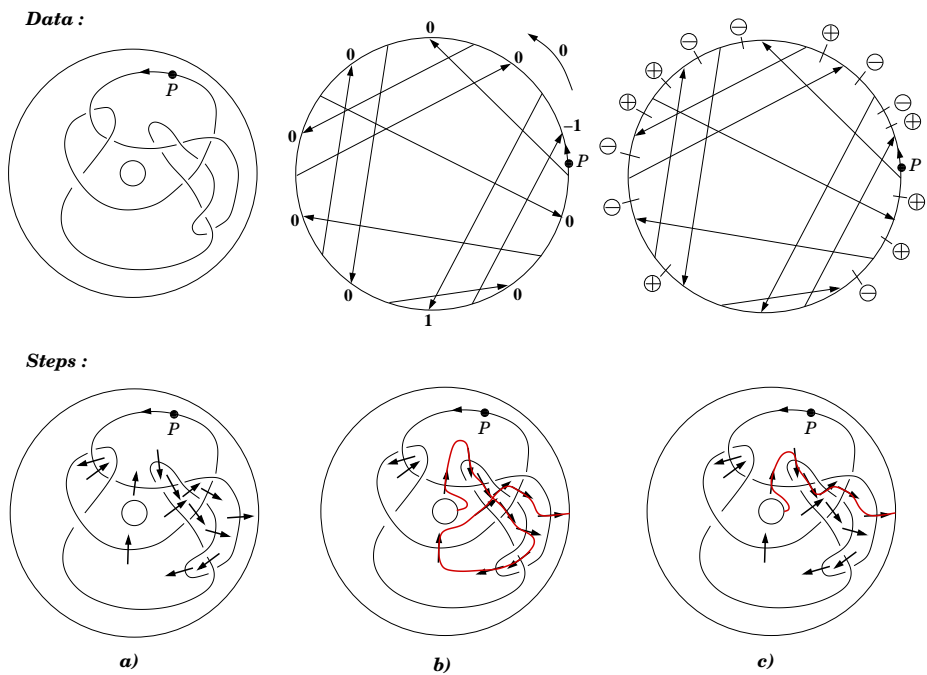,scale=1}
\caption{Steps of the proof of Lemma \ref{key}}\label{10}
\end{figure}

Now assume that $D$ is full. Draw the markings of $\G$ on $D$. Then replace each of them by a little arrow transverse to the knot diagram, that indicates the way a section should locally behave so as to give the marking (step \textit{a)} on Fig.\ref{10}). What we need to prove is that there is a path joining the two ends of $\RR\times\SS$, without self-intersection (so that some ambient isotopy can make it into a section $\RR\times\left\lbrace t\right\rbrace$), and meeting $D$ only at the places and with directions indicated by the arrows ($\mathcal{G}^\prime$ being obtained by removing all the markings left away from that path).

We start to draw such a path $\gamma$ from the end corresponding to $-\infty$. Recall that the sum of the markings met by a loop in $\G$ is equal to its homology class in $\RR\times\SS$. Since the leftmost loop of $D$ has homology class $1$, there is at least one little arrow indicating a \enquote{way to leave} the component, and as soon as we have left it, there will be as many ways to come back as to leave again. 
\begin{note}
Here is the crucial point for $D$ to be real: if it were not, then the boundary of its complementary components would be loops in $D$, but not in $\G$ (for the two preimages of a virtual crossing are not connected by an arrow), so that we would have no control on the markings they contain.
\end{note} 
The internal loops have homology class $0$, so that each time $\gamma$ enters a component which is not the $+\infty$ end, the fact that it could come implies that there is necessarily a way to leave. So we are sure of eventually reaching $+\infty$ (step \textit{b)} on Fig.\ref{10}). Finally, if $\gamma$ ever crosses itself, just forget what happened between the two times it were at this point (step \textit{c)} on Fig.\ref{10}).
\end{proof}

We are now ready for the main theorem of this section:

\begin{theorem}\label{real}
Two full knot diagrams with the same decorated Gauss diagram are isotopic to each other.
\end{theorem}

\begin{proof}
Let $\G$ be a refinement of the common decorated Gauss diagram $G$, with \textit{minimal} number of markings - notice that this number must be positive by Lemma \ref{full}. Then by Lemma \ref{key} both knot diagrams write as the closure of a tangle with $T-$diagram $\G$.
So all we need to prove is that there is only one way to draw a real tangle diagram representing $\G$ in the previous conditions.

First let us show that the connected components of such a tangle diagram are uniquely determined, so that the only possible choice left consists in their relative position in $\RR\times\left[ 0,1\right]$.

We pick a marking in $\G$: it must correspond to the beginning of a strand. From there we follow the orientation of $\G$, collecting every piece of information we meet and using them to draw, step by step, a neighborhood of the strand (see Fig.\ref{11}).

\begin{figure}[h!]
\centering 
\psfig{file=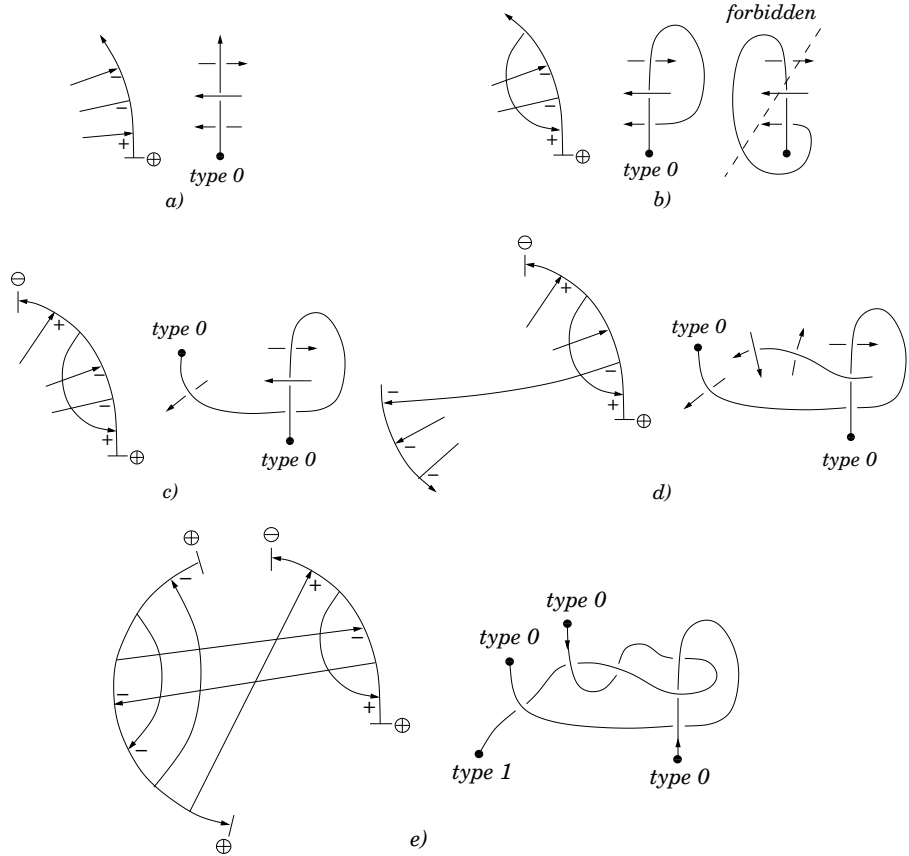,scale=1}
\caption{The connected components of a real tangle are uniquely determined}\label{11}
\end{figure}

$\blacktriangleright$ When we encounter the endpoint of an arrow, its direction and sign completely determine the local picture (step \textit{a}).

$\triangleright$ If we meet the second endpoint of some arrow, it means we have to join a piece of diagram already in the picture, and there is at most one way to do it: indeed, the ends of every strand must stay in the unbounded region of the plane, so that at the end we can glue them to the boundary of $\RR\times\left[ 0,1\right]$  (step \textit{b}). This point makes the crucial difference between Gauss diagram theories in $\RR^2$ and $\RR\times\SS$.

$\triangleright$ When we finally meet a marking again, it is the end of the strand and we stop here (step \textit{c}).

$\blacktriangleright$ If there are any, we pick an arrow which we have met at exactly one endpoint, and start over from the other one to extend the corresponding local incomplete picture (step \textit{d}).

When this is all over, we have drawn what must be a connected component of any tangle representing $\G$, without making any choice (step \textit{e}).

Each univalent vertex of these components has a \textit{type}, $0$ or $1$, according to whether it is meant to be glued to either of the two sides of $\RR\times\left[ 0,1\right]$. This side information is contained in $\G$, through the orientation and the sign of the markings. Let us call \textit{mates} two univalent vertices that are meant to be identified when we finally close the tangle.

We claim that any of the components must contain both types of vertices. Otherwise, consider a tangle whose $T-$diagram is $\G$ (we know there is at least one), then $S^\prime$ on Fig.\ref{n+k+2} shows how to obtain a new refinement of $G$ with less markings than $\G$ - which is a contradiction.

So our components look like bowels, as pictured in Fig.\ref{n+k+2}, and it remains to show that we can read on $\G$ in what order they shall be.

\begin{figure}[h!]
\centering 
\psfig{file=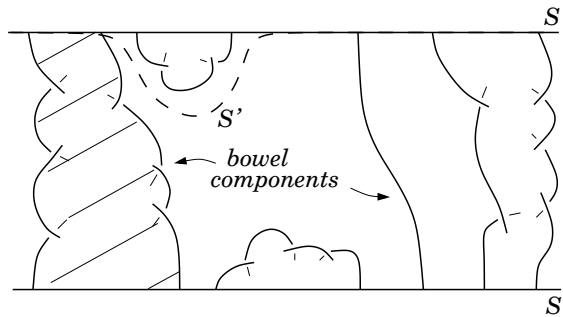,scale=1}
\caption{The connected components of a tangle}\label{n+k+2}
\end{figure}

First, the leftmost of them is uniquely determined by the property that its left \enquote{boundary path} joins two mates (it is meant to become the leftmost loop of the knot). Indeed, if there were two or more components with this property, the tangle would close into a disconnected diagram, and certainly not into a knot.

Now assume that we have been able to determine which are the $k$ leftmost components in a unique way. Look at the bottom ends of the picture they form.

$\triangleright$ Take the leftmost of them whose mate is not already in the picture: the component containing its mate is necessarily the next we should draw.

$\triangleright$ If all bottom ends already have their mate, look at the upper ends and repeat the same procedure.

$\triangleright$ If all upper ends also have their mate, then it means that the two rightmost ends in the picture are mates, and belong to one and the same component: this characterizes the rightmost loop of the knot, so the picture is actually complete.

\end{proof}

As a corollary, we see that \textit{any} function defined on full knot diagrams may theoretically be computed on decorated Gauss diagrams. In particular:

\begin{corollary}\label{wind} The Whitney index is a well-defined \emph{integer-valued} function on the set of decorated Gauss diagrams of real knots with at least one non zero valuation.
\end{corollary}

Recall that in the classical case, the Whitney index is only defined modulo $2$ on the set of Gauss diagrams of real knots.

\section{A look at closed braid diagrams}
\label{sec:braids}

In this section we use $T-$diagrams to detect configurations of arrows which may not happen in the decorated Gauss diagram of a closed braid. The basic idea is that in a closed braid diagram, any nontrivial loop which respects the orientation must have positive homology class in $\RR\times\SS$.

For example, it is shown in \cite{F2} that the configuration on Fig.\ref{ex}a may not happen in a closed braid diagram. Another proof of this fact consists in finding a refinement, as pictured on Fig.\ref{ex}b, to see that the red loop has homology class $0$ in $\RR\times\SS$ since it avoids all the markings, though it always respects the orientation of the circle.

\begin{figure}[h!]
\centering 
\psfig{file=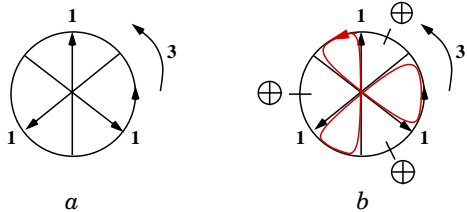,scale=1}
\caption{This configuration is not braid-admissible}\label{ex}
\end{figure}

\begin{definition} A decorated Gauss diagram $G$ is called \textit{braid-admissible}, or simply \textit{admissible} (\textsl{resp.} \textit{weakly admissible}), if every non-trivial loop in $G$ has positive (\textsl{resp.} non-negative) homology class in $\RR\times\SS$ (see Lemma \ref{hom}) as soon as it always respects the orientation of the circle - though it may go along the arrows in any direction.
\end{definition}

The rest of this section is devoted to the proof of the following:
\begin{theorem}
A decorated Gauss diagram is braid-admissible if and only if it may actually be represented by a virtual closed braid.
\end{theorem}
\begin{proof} The \enquote{if} part is trivial. Propositions \ref{lev} and \ref{rep} together prove the converse for positive $T-$diagrams, and Corollary \ref{ref} shows that this is enough.
\end{proof}

We begin with a technical lemma that will be the key-point to the last part of the proof (Proposition \ref{pos}).

Set $M=\ZZ^n=\bigoplus\ZZ.e_i$ and call $\varepsilon\in M$ a \textit{unit} if for each $i$, $\varepsilon_i \in \left\lbrace -1,0,1\right\rbrace$. If $\varepsilon$ and $\eta$ are two units, we write $\varepsilon\leq\eta$ if for all $i$, $\varepsilon_i$ is either $0$ or equal to $\eta_i$. To any $x\in M$ is associated a unit $\varepsilon(x)$ defined by $\varepsilon_i(x)=\textsl{sign}(x_i)$, with the convention $\textsl{sign}(0)=0$. Let $V$ be a submodule of $M$ and $v\in V$. We define  $$V_v=V\cap\bigoplus \varepsilon_i(v)\NN e_i.$$
There are obvious equivalences: $$\begin{array}{ccccc}

 V_v\subset V_w & \Leftrightarrow & v\in V_w & \Leftrightarrow & \varepsilon(v)\leq\varepsilon(w). \end{array} $$

We say that $V$ has property $\mathfrak{P}$ if for each $x\in V$, $V_x$ is positively generated by the units it contains - that is, any element of the former is a sum of elements of the latter, each of which may be added more than once.

\begin{lemma}\label{ext}
With the above notations, assume that $V$ has property $\mathfrak{P}$, and assume that $v_0 \in V$ has no coordinate equal to $0$. Then any group homomorphism $\phi:V\rightarrow \ZZ$ such that $\phi(V_{v_0}) \subset \NN$ extends to a homomorphism $\Phi: M\rightarrow \ZZ$ such that $\Phi(M_{v_0}) \subset \NN$.
\end{lemma}

\begin{proof} 
For the sake of simplicity, we will write $V_0$, $W_0$, \textsl{etc.}, instead of $V_{v_0}$, $W_{v_0}$, \textsl{etc.}, and $\varepsilon (k)$ will stand for the sign of $k$. Fix $i$ such that $e_i\notin V$, and set $W=V \oplus \ZZ e_i$.
The proof is in two steps. First, we show that $W$ still has property $\mathfrak{P}$; then that $\phi$ extends to $\psi: W \rightarrow \ZZ$ such that $\psi(W_0) \subset \NN$. The lemma follows by iteration of these two steps.

Let $w=v+k.e_i\in W$ with $v\in V$ and $k\in \ZZ$. Write $v$ as a sum of units $v_l$ lying in $V_v$. Note that $\varepsilon_j(w)=\varepsilon_j(v)$ for $j\neq i$.

Case 1: $\varepsilon_i(v)\in \left\lbrace 0, \varepsilon(k)\right\rbrace $. Then the $v_l$'s also lie in $W_w$, and so does the unit $\varepsilon(k).e_i$, so we are happy with: $$w= \sum v_l + \lvert k\rvert \left( \varepsilon(k).e_i\right)$$.

Case 2: $\varepsilon_i(v)=-\varepsilon(k)$. Put $L:=\left\lbrace l\mid \varepsilon_i(v_l)\neq 0\right\rbrace$. Notice that:
\[\begin{array}{lcccl} 
& \forall l\in L, & v_l\in V_v & \Longrightarrow & \varepsilon_i(v_l)=-\varepsilon(k) \\
& & & \Longrightarrow & v_{l}+\varepsilon(k).e_i \text{ is a unit of } W_v,\\

\text{and also:} & & & & \\
& & & & \\
& \forall l\notin L, & \varepsilon(v_l)\leq \varepsilon(w) & \Longrightarrow & v_l\in W_w \\
& & & \Longrightarrow & v_l \text{ is a unit of } W_w,
\end{array} \]

$\triangleright$ If $\lvert k\rvert < \sharp L$, then $\varepsilon_i(w)=\varepsilon_i(v)$, which implies that $\varepsilon(v_l)\leq \varepsilon(w)$ and $v_l$ is a unit of $W_w$ even for $l\in L$, and if we pick any $k$ elements of $L$, say $l_1, \ldots, l_k$, then: $$w = \sum_{j=1}^k\left( v_{l_j}+\varepsilon(k).e_i\right)  + \sum_{l\in L\setminus \left\lbrace l_1,\ldots l_k\right\rbrace} v_l + \sum_{l\notin L} v_l.$$ 

$\triangleright$ If $\lvert k\rvert = \sharp L$, then simply: $$w = \sum_{l\in L}\left( v_{l}+\varepsilon(k).e_i\right) + \sum_{l\notin L} v_l.$$ 

$\triangleright$ If $\lvert k\rvert > \sharp L$, then $\varepsilon_i(w)=\varepsilon(k)$, which means $\varepsilon(k).e_i$ is a unit of $W_w$, and: $$w = \sum_{l\in L}\left( v_{l_j}+\varepsilon(k).e_i\right) + \sum_{l\notin L} v_l + \left( \lvert k\rvert-\sharp L\right) \left( \varepsilon(k).e_i\right) .$$ 
In any case we have a positive decomposition of $w$ along units of $W_w$.
So we have proved that $W$ still has property $\mathfrak{P}$.
\newline

We now want to set $\psi_{\mid V}=\phi$ and give a value to $\psi(e_i)$ so that $\psi\left( W_0\right) \subset \NN$. First, notice that since by assumption $v_0$ has no zero coordinate, either of $e_i$ and $-e_i$ lies in $W_0$, namely $\varepsilon_i(v_0).e_i$, so we need to ensure that $\psi(\varepsilon_i(v_0).e_i)\geq 0$. Besides this, every element $w=v+k.e_i\in W_0$ with $k\neq 0$ gives a condition, namely $\psi(e_i) \geq -\frac{1}{k}\psi(v)$ if $k>0$, $\psi(e_i) \leq -\frac{1}{k}\psi(v)$ if $k<0$.
Let us look first at the conditions yielded by $k=\pm 1$. If we set:
\[ \begin{array}{l}
\mathcal{K}_{-1}= \left\lbrace \phi(v) \mid v\in V, \hspace{0.1cm} v-e_i \in W_0\right\rbrace \text{, and} \\
\mathcal{K}_{+1}= \left\lbrace -\phi(v) \mid v\in V, \hspace{0.1cm} v+e_i \in W_0\right\rbrace, 
\end{array} \]
then all the $k=\pm 1$-conditions reduce to: $$\textsl{sup}\, \mathcal{K}_{+1} \leq \psi(e_i) \leq \textsl{inf}\, \mathcal{K}_{-1}.$$ 

The assumption that $v_0$ has no zero coordinate ensures that $\left\lbrace v_0\pm e_i \right\rbrace \subset W_{0}$, so that $\mathcal{K}_{-1} \text{ and } \mathcal{K}_{+1}$ are not empty, \ie $\, -\infty < \textsl{sup}\, \mathcal{K}_{+1}$ and $\textsl{inf} \,\mathcal{K}_{-1} < +\infty$.

Assume now that $\varepsilon_i(v_0)=+1$, and set $\psi(e_i)=\textsl{inf} \,\mathcal{K}_{-1}$. Were it $-1$, we would set instead $\psi(e_i)=\textsl{sup} \,\mathcal{K}_{+1}$ and everything would work exactly the same way.
Since $\varepsilon_i(v_0)>0$:
$$\begin{array}{cclc}\left[ v\in V,\hspace{0.1cm} v-e_i \in W_0\right] & \Longrightarrow & v\in V_0 &\hphantom{15cm} \\
& \Longrightarrow & \phi(v)\geq 0. &\hphantom{15cm} \\
\end{array}$$
This shows that $\psi(e_i)=\textsl{inf} \,\mathcal{K}_{-1}$ is non-negative: the condition $\psi(\varepsilon_i(v_0).e_i)\geq 0$ is filled.

Let $w_1=v_1+e_i$ and $w_2=v_2-e_i$ lie in $W_0$. Then:
$$\begin{array}{cccl}
\hphantom{15cm} & v_1+v_2=w_1+w_2 & \Longrightarrow & v_1+v_2\in W_0\cap V=V_0 \\
\hphantom{15cm} & & \Longrightarrow & -\phi(v_1)\leq \phi(v_2).

\end{array}$$
So we have proved that: $$\infty<\textsl{sup}\, \mathcal{K}_{+1}  \leq \psi(e_i)=\textsl{inf} \,\mathcal{K}_{-1} <+\infty,$$
and $\psi(e_i)$ is an integer with the required sign, satisfying all $k=\pm 1$-conditions.
\newline

Now we claim that these elementary conditions are enough for all the others to hold. Indeed, let $k$ be any integer and $v\in V$ such that $w=v+k.e_i\in W_0$. Throughout the proof that $W$ satisfies the property $\mathfrak{P}$, we actually showed that such a $w$ is a sum of units of $W_w$ of either of the forms $v_l+e_i$, $v_l-e_i$, $v_l$, $\varepsilon_i(v_0).e_i$, where the $v_l$'s stand for units of $V_v$. Since $w\in W_0\Rightarrow W_w\subset W_0$, such units have non-negative image by $\psi$, and therefrom so does $w$.
\end{proof}

\begin{remark}
There is a stronger version with no hypothesis on the coordinates of $v_0$, but we will not need it here.
\end{remark}

\begin{proposition}\label{pos} A decorated Gauss diagram has a non-negative refinement if and only if it is weakly admissible.
\end{proposition}
\begin{corollary}\label{ref}
An admissible decorated Gauss diagram always has a positive refinement.
\end{corollary}

\begin{proof} The \enquote{only if} part is trivial. Let $G$ be a weakly admissible diagram, with $n$ arrows. Recall that the \textit{edges} of $G$ are the connected components of the complementary of its arrows.
We embed $H_1(G)$ into $\ZZ^{2n}=\bigoplus \ZZ e_i$ in the following way: number the edges of $G$ from $1$ to $2n$, then send every distinguished loop in $G$ to the sum of the $e_i$'s corresponding to the edges it goes through, and send the fundamental class of the circle to the sum of all $e_i$'s.
Call this embedding $\iota$, and set $V:=\iota (H_1(G))$. Since $\i$ defines an isomorphism between $V$ and $H_1(G)$, we may set:
\[
\begin{array}{rcccl}
\phi:& V & \longrightarrow & \ZZ= H_1(\RR\times\SS) &,\\
& \i (\gamma ) & \longmapsto & \left[ \gamma \right] &
\end{array}
\]
Now all we need to do is extend the map $\phi$ into an element of $\bigoplus \NN e_i^*$: the coefficient against $e_i^*$ will indicate how many markings we shall put on the $i$-th edge of $G$.

First, we see that $V$ satisfies the property $\mathfrak{P}$: indeed, this is a general fact about $1-$dimensional cellular spaces. Shrink every arrow of $G$ to a point (vertex), so that the $e_i$'s fully describe its cellular structure. Then let $v\in V$ write as $\i \left(x= \sum  x_i e_i\right) $. Let $\gamma$ be a path in $G$ corresponding as a $1-$chain to $\varepsilon_{i_1}(x) e_{i_1}$, where $i_1$ is such that $x_{i_1}\neq 0$. If the two ends of $e_{i_1}$ were joined by an arrow, then $\gamma$ is a loop and we stop here. Otherwise, since $x$ is a cycle, there must be some $i_2$ such that $\varepsilon_{i_2}(x) e_{i_2}$ is a path that starts where $\gamma$ ends. So we put $\gamma=\varepsilon_{i_1}(x) e_{i_1}+\varepsilon_{i_2}(x) e_{i_2}$, and iterate this process until $\gamma$ meets some vertex $A$ for the second time. Then we forget what happened before $\gamma$ first met $A$. What remains is a loop, whose image by $\iota$ is a unit of $V_v$. Repeating this with $v-\iota(\gamma)$, and so on, we split $v$ as a sum of units of $V_v$.

By the same process, we see that an element of $V\cap \NN^{2n}$ corresponds through $\i$ to the sum of the fundamental classes of loops in $G$ that can be chosen so as to respect orientation. So if we set $v_0=\sum e_i \in V$, then the weak admissibility of $G$ implies that $\phi$ takes $V_0=V\cap \NN^{2n}$ into $\NN$. Since $v_0$ has no zero coordinate, lemma \ref{ext} applies and gives the required extension.
\end{proof}

For the sequel, we will need a combinatorial tool highly inspired from the topology of braids:

\begin{definition} Let $D$ be a positive $T-$diagram. We say that an arrow has \textit{level $1$} if each of its endpoints is directly preceded by a (positive) marking. Remove every arrow of level $1$. Those of level 1 in the new diagram are said to have \textit{level $2$} in $D$. By induction we define the arrows of \textit{level $k$} in $D$. Of course some arrows may have no level at all.
\end{definition}

\begin{proposition}\label{lev} A positive $T-$diagram is admissible if and only if each of its arrows has a well-defined level.
\end{proposition}

\begin{proof} Note that if there is no arrow, then the diagram is admissible and the lemma is true. We assume from now on that $D$ has at least one arrow.

For positive diagrams, being admissible is equivalent to satisfying the property that every nontrivial 
and orientation respecting loop meets at least one marking. So assume that each arrow has a level, but that some loop fails to meet any marking. There must be at least one arrow involved in that loop. If not, then it would go all the way around the whole circle, and meet markings since the diagram is positive: that is a contradiction.
Remove the arrows which are not involved in the loop. Those remaining had a well defined level, so they still have one in the new diagram, and so at least one of them, say $A$, has level $1$ in the new diagram. It means that each endpoint of $A$ is directly preceded by a marking. Since our loop goes along $A$, and respects the orientation of the circle, it must meet one of them: that is a contradiction.

Conversely, assume that some arrow has no level, and remove all of those which have one. In this new diagram, pick a marking which is not directly followed by another, i.e. which is directly followed by the endpoint of an arrow. We start our loop at this endpoint. Go along the arrow. The other endpoint cannot be also preceded by a marking, or the arrow would have a level. So go back along the circle until we find the endpoint of an arrow that is directly preceded by a marking. Then iterate those two steps: we must loop, since there is a finite number of arrows, and the loop we created avoids every marking.
\end{proof}

\begin{proposition}\label{rep} A positive $T-$diagram can be represented by a virtual closed braid if and only if each of its arrows has a well-defined level.
\end{proposition}

\textbf{Note:} it is to be understood that an isotopy of a braid diagram always stays within the set of braid diagrams - and as usual, \mbox{involves no Reidemeister move.}

\begin{proof} Every real crossing of a virtual braid which may be isotoped into the lowest corresponds to an arrow of level $1$ in the associated $T-$diagram. Removing these arrows amounts to replacing the real crossings by virtual ones. Until there is no more real crossing, there will always be a lowest one. This proves the \enquote{only if} part.

Assume that each arrow of a positive diagram has a level. Let $k$ be the global marking of the circle, and let $l$ be the maximal level of the arrows.
Cut the circle at every marking, so as to obtain a diagram based on a union of $k$ segments, and embed it into $\RR\times \left[0,l+1\right]$, 
in such a way that the segments are oriented from bottom to top, and such that each arrow of level $i$ is contained in $\RR\times \left\lbrace i\right\rbrace $. 
The fact that each arrow has a level is equivalent to the existence of such an embedding. Now at each level make the strands cross each other as indicated 
by the arrows, by a homotopy that keeps the $i+1/2$-levels untouched. Declare virtual every additional crossing needed to do that. Finally, add to the top a totally virtual braid corresponding to the permutation defined by the way the strands were originally glued together in the circle, so that the resulting braid closes into a knot with the required $T-$diagram.
\end{proof}

Institut de Mathematiques de Toulouse

Universite Paul Sabatier et CNRS (UMR 5219)

118, route de Narbonne

31062 Toulouse Cedex 09, France

mortier@math.ups-tlse.fr

\end{document}